\documentclass{article}
\usepackage[utf8]{inputenc}

\usepackage{amsmath}
\usepackage{amssymb}
\usepackage{amsthm}
\usepackage{tikz}
\usetikzlibrary{matrix,arrows.meta} 

\def\Gal{\text{Gal}}
\newtheorem{theorem}{Theorem}[section]
\newtheorem{proposition}[theorem]{Proposition}
\newtheorem{lemma}[theorem]{Lemma}
\newtheorem{corollary}{Corollary}[theorem]

\title{Class Group Relations in a Function Field Analogue of ${\mathbb Q}(\zeta_p, \sqrt[p]{n})$}
\author{Steven Reich}
\date{November 2020}

\begin{document}

\maketitle

\section{Introduction}

Let $p$ be an odd prime and $k={\mathbb F}_p(T)$. Let $\lambda$ be a non-zero root of $x^p+Tx=0$, and $\gamma$ be a root of $x^p+Tx=P(T)$, where $P(T) \in k$ but $P(T) \neq Q(T)^p + TQ(T)$ for any $Q(T) \in k$. We have the following lattice of fields\footnote{The edge labels indicate the Galois groups of the corresponding extensions, which will be described in Section \ref{proof}.}:

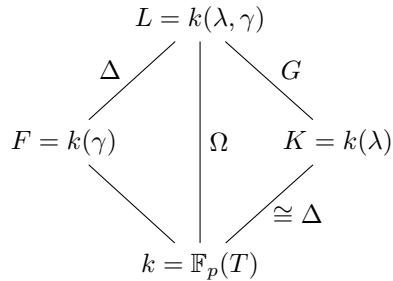
{\centering 
\begin{tikzpicture}
\matrix (m) [matrix of math nodes,row sep=3em,column sep=0em,minimum width=2em,ampersand replacement=\&]
{
\& L = k(\lambda, \gamma) \& \\
F = k(\gamma) \& \& K = k(\lambda) \\
\& k = {\mathbb F}_p(T) \& \\
};
\path[-]
(m-2-1) edge node [pos=0.4, above] {$\Delta$ \, \,} (m-1-2)
(m-2-3) edge node [pos=0.4, above] {\, \, $G$} (m-1-2)
(m-3-2) edge node {} (m-2-1)
(m-3-2) edge node [right] {$\Omega$} (m-1-2)
(m-3-2) edge node [pos=0.6, below] {\, \, \, $\cong \Delta$} (m-2-3)
;
\end{tikzpicture}

}

Our main objective is to prove the following theorems about $Cl_L^0$, the group of degree-0 divisor classes\footnote{We take the \emph{divisor group} of a function field to mean the free abelian group indexed by its primes. Its quotient by the principal divisors (those which represent an element of the field) is the \emph{divisor class group}. The subgroup of elements with total exponent 0 is then the \emph{degree-0 divisor class group}.} of $L$:

\begin{theorem}\label{main}
$Cl_L^0$ is isomorphic to a $(p-1)$-st power of a finite abelian group.
\end{theorem}

\begin{theorem}\label{main2}
The class numbers of $L$ and $F$ are related via $h_L=h_F^{p-1}$ (and in fact this holds when $p$ is a prime power).
\end{theorem}

When the $\ell$-rank of $Cl_F^0$ is 1 for every prime $\ell$ dividing $h_F$ (e.g. when $h_F$ is square-free), these combine to say that $Cl_L^0=(Cl_F^0)^{p-1}$. However, this is not known in general.

In Section \ref{pre}, we motivate these theorems and provide some background on the fields being considered, and the remaining sections are devoted to the proofs.

\section{Preliminaries}\label{pre}

\subsection{Number field antecedents}
Theorem \ref{main} is an analogue of a recent result of Schoof in the number field setting. In \cite{schoof2020ideal}, he proves the following:

\begin{theorem}\label{schoof}
Let $p>2$ be a regular prime and $n \in {\mathbb Z}$ not a $p$-th power. Suppose that all prime divisors $l \neq p$ of $n$ are primitive roots mod $p$. Then the ideal class group $Cl_L$ of $L = {\mathbb Q}(\zeta_p, \sqrt[p]{n})$ and the kernel of the norm map $N_{L/{\mathbb Q}(\zeta_p)}$ fit into the exact sequences
\[ 0 \to V \to \ker(N_{L/{\mathbb Q}(\zeta_p)}) \to A^{p-1} \to 0 \text{ and} \]
\[ 0 \to \ker(N_{L/{\mathbb Q}(\zeta_p)}) \to Cl_L \to Cl_{{\mathbb Q}(\zeta_p)} \to 0, \]
where $A$ is a finite abelian group and $V$ is an ${\mathbb F}_p$-vector space of dimension at most $\left(\frac{p-3}{2}\right)^2$. In particular, if $\# Cl_{{\mathbb Q}(\zeta_p)} = 1$, then $Cl_L / V$ is a $(p-1)$-st power of a finite abelian group.
\end{theorem}

Theorem \ref{main2} is inspired by one proved by Honda \cite{honda1971pure}:

\begin{theorem}\label{honda}
Let $F={\mathbb Q}(\sqrt[3]{n})$, and $L={\mathbb Q}(\sqrt[3]{n}, \zeta_3)$ be its normal closure. Then $h_L = h_F^2$ or $h_L = \frac{1}{3} h_F^2$.
\end{theorem}

\subsection{Function field analogue of ${\mathbb Q}(\zeta_p)$}

We return now to considering extensions of $k={\mathbb F}_p(T)$. Let $\Lambda$ denote the roots (in $\bar{k}$, an algebraic closure) of $x^p+Tx=0$. Fixing a nonzero root $\lambda$, we see that $\Lambda=\{m\lambda \, | \, m \in {\mathbb F}_p \} \cong {\mathbb F}_p^+$, and that $k(\Lambda)=k(\lambda)$ is a degree $p-1$ cyclic extension of $k$. These definitions should be reminiscent of those for the $p$-th roots of unity $\mu_p = \{ \zeta_p^i \, | \, i \in {\mathbb F}_p \}$, and in fact they are part of a rich theory of `cyclotomic' function fields first developed by Hayes \cite{hayes1974explicit} based on work by Carlitz \cite{carlitz1938class}. For a more thorough discussion of the connection, we refer to \cite{goss1983arithmetic}.

We note two properties which will be important to the proofs of Theorems \ref{main} and \ref{main2}, and which contribute to their comparative simplicity vis-\`a-vis Theorems \ref{schoof} and \ref{honda}. The first is that $k(\Lambda) = k(\lambda)$ has genus 0 (since $T = -\lambda^{p-1}$), and therefore its degree-0 divisor class group is trivial. The second is that there are exactly two primes which ramify in $k(\lambda) / k$, both of which are totally ramified: the prime $(T) = (\lambda)^{p-1}$, and the distinguished infinite prime $(1/T) = (1/\lambda)^{p-1}$. This is fairly easy to see in our case from the discriminant, but also applies to a more general class of such extensions \cite{hayes1974explicit}.

\subsection{Function field analogue of ${\mathbb Q}(\sqrt[p]{n})$}

Now let $\gamma$ denote a root of $x^p+Tx=P(T)$, where $P(T) \in k$ but $P(T) \neq Q(T)^p + TQ(T)$ for any $Q(T) \in k$. Then $\gamma + \lambda$, $\gamma + 2\lambda$, $\dots$, $\gamma + (p-1)\lambda$ are the other roots, and the fields $k(\gamma + i\lambda)$ are conjugate degree-$p$ extensions of $k$ with Galois closure $k(\lambda, \gamma)$. Continuing the analogy of the previous section, this construction parallels that of ${\mathbb Q}(\zeta_p^i\sqrt[p]{n})$ and its Galois closure, ${\mathbb Q}(\zeta_p^i, \sqrt[p]{n})$.

By discriminant considerations, the only primes that possibly ramify in \linebreak $k(\lambda, \gamma) / k(\lambda)$ are those lying above $(T)$ and $(1/T)$. Therefore every prime of $k$ which ramifies in $k(\lambda, \gamma)/k$ is totally ramified in $k(\lambda)/k$.

\section{Proof of Theorem \ref{main}}\label{proof}

As in the introduction, we set $k={\mathbb F}_p(T)$, $K=k(\lambda)$, $F=k(\gamma)$, and $L=k(\lambda, \gamma)$. Define $\Omega = \Gal(L/k)$, $G = \Gal(L/K)$, and $\Delta = \Gal(L/F) \cong \Gal(K/k)$. These groups have the presentations
\begin{itemize}
    \item $G = \langle \tau \, | \, \tau^p=1 \rangle \cong {\mathbb F}_p^+$,
    \item $\Delta = \langle \sigma \, | \, \sigma^{p-1}=1 \rangle \cong {\mathbb F}_p^\times$,
    \item $\Omega = \langle \sigma, \tau \, | \, \sigma^{p-1}=1, \tau^p=1, \sigma \tau \sigma^{-1} = \tau^{\omega(\sigma)} \rangle$,
\end{itemize}

where $\omega: \Delta \to {\mathbb F}_p^\times$ is the cyclotomic character defined by $\sigma(\lambda) = \omega(\sigma) \lambda$. Refer to the field diagram in the introduction for a depiction of these relationships.

Naturally, there is a Galois action of $\Omega$ on $Cl_L^0$. The norm element $N_G = \sum_G \tau$ of $G$ gives a map $Cl_L^0 \to Cl_L^0$ that factors through $Cl_K^0$, which is trivial. Thus $Cl_L^0$ is a module over the group ring ${\mathbb Z}[\Omega]/(N_G)$, or alternately a module over ${\mathbb Z}[G]/(N_G)$ with a twisted action of $\Delta$ (by which we mean that $\Delta$ acts on $Cl_L^0$ in a way that is consistent with the action of $\Delta$ on ${\mathbb Z}[G]/(N_G)$). Now, ${\mathbb Z}[G]/(N_G) \cong {\mathbb Z}[\zeta_p]$ as a $\Delta$-module (because $N_G$ is the $p$-th cyclotomic polynomial evaluated at $\tau$), so we may freely apply standard facts about $\zeta_p$ to $\tau$.\footnote{We opt to keep the notation in terms of $\tau$ rather than $\zeta_p$, to maintain coherence with the function field setting.}

We are now ready to develop the proof of Theorem \ref{main}. We proceed by separately considering the $p$ part and the non-$p$ part of $Cl_L^0$.

\subsection{The non-$p$ part of $Cl_L^0$}

In this section, let $M$ denote the non-$p$ part of the degree-0 divisor class group of $L$. The following proposition describes the $\Delta$-module structure of $M$.

\begin{proposition}\label{nonp}
The map \[ \varphi: M^\Delta \otimes_{\mathbb Z} {\mathbb Z}[G]/(N_G) \to M \] given by $\sum_i m_i \otimes [\tau^i] \mapsto \sum_i \tau^i m_i$ is an isomorphism of $\Delta$-modules. 
\end{proposition}

\begin{proof}
Suppose first that $\sum_i \tau^i m_i = 0$ (and note that since $N_G$ acts trivially, this sum can be assumed to be over $1 \leq i \leq p-1$). Then $\sum_i \tau^{i\omega(\sigma)} m_i = 0$ for all $\sigma$, and thus for $1 \leq j \leq p-1$,

\begin{align*}
0 &= \sum_{\sigma \in \Delta} \tau^{-j\omega(\sigma)} (1-\tau^{j\omega(\sigma)}) \sum_i \tau^{i\omega(\sigma)} m_i \\
&= \sum_i \sum_{\sigma \in \Delta} (1-\tau^{j\omega(\sigma)}) \tau^{(i-j)\omega(\sigma)} m_i.
\end{align*} 
$\sum_\sigma (1-\tau^{j\omega(\sigma)}) \tau^{(i-j)\omega(\sigma)}$ acts as  $p-1 + (1 - N_G) = p$ for $i=j$, and as 0 for $i \neq j$, so this says $pm_j = 0$, and thus $m_j=0$, for all $j$ (since $M$ has order prime to $p$). Therefore $\varphi$ is injective.

Now suppose $m \in M$. Then for any $i$, $N_\Delta (\tau^i m) = \sum_\sigma \tau^{i\omega(\sigma)} \sigma(m) \in M^\Delta$, and accordingly, 
\begin{align*}
&\sum_{i=1}^{p-1} \tau^{-i\omega(\sigma')}(1-\tau^{i\omega(\sigma')}) \sum_\sigma \tau^{i\omega(\sigma)} \sigma(m) \\
= &\sum_\sigma \sum_{i=1}^{p-1} (1-\tau^{i\omega(\sigma')}) \tau^{i(\omega(\sigma)-\omega(\sigma'))} \sigma(m) \in im(\varphi) \text{ for all } \sigma' \in \Delta.
\end{align*}
$\sum_i (1-\tau^{i\omega(\sigma')}) \tau^{i(\omega(\sigma)-\omega(\sigma'))}$ acts as $p$ for $\sigma' = \sigma$ and as 0 for $\sigma' \neq \sigma$, so this says that each $p\sigma'(m)$, and in particular $pm$, is in $im(\varphi)$. Therefore $\varphi$ is surjective (again using that $M$ has order prime to $p$).
\end{proof}

Ignoring the module structure gives $M \cong (M^\Delta)^{p-1}$ as abelian groups, which settles the non-$p$ part of Theorem \ref{main}.

\subsection{The $p$ part of $Cl_L^0$}

From this section forward, $M$ will denote the $p$ part of the degree-0 divisor class group of $L$. Having only $p$-power torsion allows us to strengthen the ${\mathbb Z}[G]/(N_G)$-module structure previously described to a ${\mathbb Z}_p[G]/(N_G)$-module structure, still with the twisted $\Delta$-action as before. $A = {\mathbb Z}_p[G]/(N_G)$ is a discrete valuation ring, and its maximal ideal is generated by $\tau-1$ (just as $\zeta_p-1$ generates the maximal ideal of ${\mathbb Z}_p[\zeta_p]$). Furthermore, $(\tau-1)^{p-1} = (p)$ as ideals of $A$, which will be key to this section's results. 

As before, $\Delta$ acts on $\tau$, and thus on $\tau-1$, by the cyclotomic character $\omega: \Delta \to {\mathbb F}_p^\times$. We have a filtration of ideals
\[ A \supset (\tau-1)A \supset (\tau-1)^2A \supset \dots \]
with successive quotients ${\mathbb F}_p$, ${\mathbb F}_p(\omega)$, ${\mathbb F}_p(\omega^2)$, $\dots$, where $X(\omega^i)$ denotes that the default action of $\Delta$ on the module $X$ is twisted by the character $\omega^i$.

We now prove two results which characterize the structure of $A$-modules with particularly `nice' $\Delta$-action:

\begin{lemma}
\emph{\cite[Prop.~3.1]{schoof2020ideal}} Let $M'$ be a finite $A$-module with twisted $\Delta$-action. Then $\Delta$ acts trivially on $M'/(\tau-1)M'$ if and only if there exist $n_1, n_2, \dots, n_t \geq 1$ such that
\[ M' \cong \bigoplus_{i=1}^t A/(\tau-1)^{n_i}A. \]
\end{lemma}

\begin{proof}
Suppose first that the isomorphism holds. Then $M'/(\tau-1)M'$ is a direct sum of $A/(\tau-1)A \cong {\mathbb F}_p$ terms with trivial $\Delta$ action. 

Conversely, suppose that $\Delta$ acts trivially on $M'/(\tau-1)M'$. Then the map $M'^\Delta \to (M'/(\tau-1)M')^\Delta = M'/(\tau-1)M'$ is surjective (its cokernel is the first $\Delta$-cohomology group of $(\tau-1)M'$, which is trivial because $\Delta$ and $M'$ have comprime orders). This says that there are $\Delta$-invariant elements $v_1, \dots, v_t$ which generate $M'$ over $A$, i.e. that there is a surjective map from $A^t \to M'$ taking 1 in the $i$-th coordinate to $v_i$. By the finiteness of $M'$, this descends to a surjective map
\[ \varphi: \bigoplus_{i=1}^t A/(\tau-1)^{n_i}A \to M'. \]
If $\varphi$ is not injective, there is a nonzero element $x$ in the kernel, and we may assume $x = (x_1, \dots, x_t) \in \bigoplus_i (\tau-1)^{n_i-1}A/(\tau-1)^{n_i}A$. Now, $\Delta$ acts on $x$ by $\omega^m$ for some $m$, but by $\omega^{n_i-1}$ on each of these summands, so $x_i$ is zero unless $n_i-1=m+k_i(p-1)$ for some $k_i$. Reordering if necessary, let the first $s$ coordinates of $x$ be exactly those which are nonzero, and choose $n_1$ to be minimal among $n_1, \dots, n_s$. For $i=1, \dots, s$, define $\mu_i$ such that $x_i = (\tau-1)^m p^{k_i} \mu_i$, and $m_i \in {\mathbb Z}$ such that $m_i \equiv \mu_i \mod (\tau-1)^N$, where $N$ is the maximum of the $n_i$. Notice that $\mu_i$, and thus $m_i$, is a unit in $A$.

We are now able to construct a new map
\[ \varphi': A/(\tau-1)^{n_1-1}A \oplus \bigoplus_{i=2}^t A/(\tau-1)^{n_i}A \to M', \]
which takes the basis vector $e_1 = (1, 0, \dots, 0)$ to $\sum_{i=1}^s m_i p^{k_i-k_1} v_i$, and $e_i$ to $v_i$ for $i \neq 1$. This is well-defined because
\begin{align*}
\varphi'((\tau-1)^m p^{k_1} e_1) &= \sum_{i=1}^s (\tau-1)^m p^{k_i} m_i v_i \\
&= \sum_{i=1}^s (\tau-1)^m p^{k_i} \mu_i v_i = \sum_{i=1}^s x_i v_i = \varphi(x) = 0.
\end{align*}
Furthermore, $\varphi'$ is surjective because, by the surjectivity of $\varphi$, $m_1 v_1 \in im(\varphi')$, and $m_1$ is invertible. If $\varphi'$ is not injective, we repeat this procedure until we reach a map that is, at which point we will have found a direct sum of finite quotients of $A$ which is isomorphic to $M'$.
\end{proof}

\begin{proposition}\label{pstruct}
Let $M'$ be a finite $A$-module with twisted $\Delta$-action, such that $\Delta$ acts trivially on $M'/(\tau-1)M'$ and by $\omega^{-1}$ on $M'[\tau-1]$. Then there exists a finite abelian $p$-group $H$ such that
\[ M \cong H \otimes_{{\mathbb Z}_p} A. \]
\end{proposition}

\begin{proof}
Suppose $M' \cong A/(\tau-1)^n A$ for some positive integer $n$. Then $M'[\tau-1]=(\tau-1)^{n-1}A/(\tau-1)^n A \cong {\mathbb F}_p(\omega^{n-1})$. By assumption, this requires $n=(p-1)m$ for some positive integer $m$, whereby $A/(\tau-1)^n A \cong A/p^m A \cong {\mathbb Z}/p^m{\mathbb Z} \otimes_{{\mathbb Z}_p} A$.

By the previous lemma, a general $M'$ satisfying the condition on $M'/(\tau-1)M'$ is a direct sum of such ${\mathbb Z}/p^m{\mathbb Z} \otimes_{{\mathbb Z}_p} A$, proving the proposition.
\end{proof}

It remains to show that this proposition can be applied to (a twist of) $M$. We will achieve this by exploring the Galois cohomology of $Cl_L^0$ and related objects.

\subsection{Galois cohomology of $Cl_L^0$}
In this section, $\hat{H}^i(X)$ is the $i$-th Tate $G$-cohomology group of $X$. We fix the notations
\begin{itemize}
    \item $P_L$, the principal $L$-divisors
    \item $D_L^0$, the $L$-divisors of degree 0
    \item ${\mathbb I}_L^0$, the ideles of total valuation 0
    \item $C_L^0 = {\mathbb I}_L^0 / L^\times$, the idele classes of total valuation 0
    \item $U_L$, the product of the local unit groups $U_{\mathfrak L}$ of $L$.
\end{itemize}

We have the following commutative diagram of $\Omega$-modules in which the rows and columns are exact:

\[\def\arraystretch{1.5}
\begin{array}{ccccccccc}
 & & 1 & & 1 & & 0 & & \\
 & & \downarrow & & \downarrow & & \downarrow & & \\
1 & \to & {\mathbb F}_p^\times & \to & L^\times & \to & P_L & \to & 0 \\
 & & \downarrow & & \downarrow & & \downarrow & & \\
1 & \to & U_L & \to & {\mathbb I}_L^0 & \to & D_L^0 & \to & 0 \\
 & & \downarrow & & \downarrow & & \downarrow & & \\
1 & \to & U_L / {\mathbb F}_p^\times & \to & C_L^0 & \to & Cl_L^0 & \to & 0 \\
 & & \downarrow & & \downarrow & & \downarrow & & \\
 & & 1 & & 1 & & 0 & &
\end{array}
\]

We will make use of various long exact sequences induced in cohomology by the above diagram in order to study $\hat{H}^i(Cl_L^0)$ for $i=-1, 0$. We remark that a $G$-cohomology group of an $\Omega$-module is an ${\mathbb F}_p[\Delta]$-module (because it is killed by $p$ and is $G$-invariant), and any map induced in cohomology by any of the maps in the diagram is $\Delta$-equivariant. Furthermore, since $G$ is a cyclic group, we have $\Delta$-isomorphisms $\hat{H}^i(X) \to \hat{H}^{i-2}(X)(\omega^{-1})$ for every $i \in {\mathbb Z}$ and $\Omega$-module $X$, given by cupping with a generator of $\hat{H}^{-2}({\mathbb Z}) \cong H_1({\mathbb Z}) \cong G \cong {\mathbb Z}/p{\mathbb Z}(\omega)$.

\begin{lemma}
$C_L^0$ and ${\mathbb F}_p^\times$ have trivial Tate $G$-cohomology.
\end{lemma}

\begin{proof}
The degree map on the idele class group gives rise to the sequence
\[ 0 \to C_L^0 \to C_L \to {\mathbb Z} \to 0. \]
This gives rise to a long exact sequence which includes
\[ H^0(C_L) \to H^0({\mathbb Z}) \to \hat{H}^1(C_L^0) \to \hat{H}^1(C_L), \]
where the first two terms are standard (i.e. non-Tate) cohomology. We have that $H^0(C_L) = C_L^G = C_K$ \cite[p.~2]{artin1968class}. The leftmost map, then, is the degree map on $C_K$ \emph{as a subgroup of} $C_L$. An idele class of $C_K$ which has valuation 1 at an inert prime and valuation 0 elsewhere maintains this property when extended to $C_L$. Since $G$ is cyclic, there are (infinitely many) primes inert in $L/K$ by the Chebotarev density theorem, and thus the map $H^0(C_L) \to H^0({\mathbb Z})$ is surjective. We also have $\hat{H}^1(C_L) = 0$ \cite[p.~19]{artin1968class}, and so $\hat{H}^1(C_L^0) = 0$.

Passing the first exact sequence to Tate cohomology and recognizing that \linebreak $\hat{H}^0(C_L) \cong {\mathbb Z}/p{\mathbb Z}$ \cite[p.~19]{artin1968class}, $\hat{H}^0({\mathbb Z}) \cong {\mathbb Z}/p{\mathbb Z}$, and $\hat{H}^{-1}({\mathbb Z}) \cong \hat{H}^1({\mathbb Z}) = 0$, we have that $\hat{H}^0(C_L^0) = 0$ as well.

As for ${\mathbb F}_p^\times$, it is a finite module of order prime to the order of $G$, and so has trivial Tate $G$-cohomology.
\end{proof}

Applying this lemma to the long exact sequences induced by the leftmost column and bottom row of the diagram immediately gives:

\begin{corollary}\label{iso}
$\hat{H}^{-1}(Cl_L^0) \cong \hat{H}^0(U_L)$ and $\hat{H}^{0}(Cl_L^0) \cong \hat{H}^1(U_L)$.
\end{corollary}

\begin{lemma}\label{action}
$\Delta$ acts trivially on $\hat{H}^1(U_L)$ and $\hat{H}^2(U_L)$.
\end{lemma}

\begin{proof}
We write $\ell$ for a prime of $k$, ${\mathfrak l}$ for a prime of $K$ above $\ell$, and  ${\mathfrak L}$ for a prime of $L$ above ${\mathfrak l}$. We have decomposition groups $\Omega_{\mathfrak L} = D({\mathfrak L}/\ell)$, $G_{\mathfrak L} = D({\mathfrak L}/{\mathfrak l})$, and $\Delta_{\mathfrak l} = D({\mathfrak l}/\ell)$. For each $i$, $\hat{H}^i(U_L)$ can be expressed as a product of local cohomology groups:

\begin{align*}
\hat{H}^i(U_L) &= \hat{H}^i(\prod_{{\mathfrak L} \text{ of } L} U_{\mathfrak L}) \\
&= \hat{H}^i(\prod_{{\mathfrak l} \text{ of } K} \bigoplus_{{\mathfrak L}|{\mathfrak l}} U_{\mathfrak L}) \\
&= \bigoplus_{\ell \text{ ram in } L} \bigoplus_{{\mathfrak l}|\ell} \hat{H}^i(\bigoplus_{{\mathfrak L}|{\mathfrak l}} U_{\mathfrak L}) \\
&= \bigoplus_{\ell \text{ ram in } L} \bigoplus_{{\mathfrak l}|\ell} \hat{H}^i(G_{\mathfrak L}, U_{\mathfrak L}),
\end{align*}
with the last equality by applying Shapiro's Lemma to $\bigoplus_{{\mathfrak L}|{\mathfrak l}} U_{\mathfrak L} = \text{Ind}_{G_{\mathfrak L}}^G U_{\mathfrak L}$. Furthermore, we saw in Section \ref{pre} that any prime that ramifies in $L$ ramifies totally in $K$, which means that $\Delta_{\mathfrak l} = \Delta$ and any action it has on $\hat{H}^1(U_L)$ is on a single summand $\hat{H}^i(G_{\mathfrak L}, U_{\mathfrak L})$. Thus for $i=1,2$, it is sufficient to show that $\Delta$ acts trivially on $\hat{H}^i(G_{\mathfrak L}, U_{\mathfrak L})$.

We look first at $i=1$. Since $G$ and $\Delta$ have coprime orders, the inflation-restriction sequence
\[ 0 \to \hat{H}^1(\Delta_{\mathfrak l}, U_{\mathfrak l}) \to \hat{H}^1(\Omega_{\mathfrak L}, U_{\mathfrak L}) \to \hat{H}^1(G_{\mathfrak L}, U_{\mathfrak L})^{\Delta_{\mathfrak l}} \to 0 \]
is exact. Now, from local class field theory, $\hat{H}^1(G_{\mathfrak L}, U_{\mathfrak L})$ is cyclic of order equal to the ramification index $e_{{\mathfrak L}/{\mathfrak l}}$ \cite[p.~9]{artin1968class}, and likewise $\hat{H}^1(\Omega_{\mathfrak L}, U_{\mathfrak L}) \cong {\mathbb Z}/e_{{\mathfrak L}/\ell}{\mathbb Z}$ and $\hat{H}^1(\Delta_{\mathfrak l}, U_{\mathfrak l}) \cong {\mathbb Z}/e_{{\mathfrak l}/\ell}{\mathbb Z}$. This forces $\hat{H}^1(G_{\mathfrak L}, U_{\mathfrak L})^{\Delta_{\mathfrak l}} = \hat{H}^1(G_{\mathfrak L}, U_{\mathfrak L})$, and so the action of $\Delta=\Delta_{\mathfrak l}$ on $\hat{H}^1(U_L)$ is trivial.

Next we take $i=2$. Again by the coprime orders of $G$ and $\Delta$, and also using that $\hat{H}^1(G_{\mathfrak L}, L_{\mathfrak L}) = 0$ by Hilbert Theorem 90, the sequence
\[ 0 \to \hat{H}^2(\Delta_{\mathfrak l}, K_{\mathfrak l}) \to \hat{H}^2(\Omega_{\mathfrak L}, L_{\mathfrak L}) \to \hat{H}^2(G_{\mathfrak L}, L_{\mathfrak L})^{\Delta_{\mathfrak l}} \to 0 \]
is exact. But local class field theory gives us that these cohomology groups are dual to the decomposition groups that define them \cite[p.~9]{artin1968class}, and so by order considerations, we must have $\hat{H}^2(G_{\mathfrak L}, L_{\mathfrak L})^{\Delta_{\mathfrak l}} = \hat{H}^2(G_{\mathfrak L}, L_{\mathfrak L})$. Since the inclusion-induced map $\hat{H}^2(G_{\mathfrak L}, U_{\mathfrak L}) \to \hat{H}^2(G_{\mathfrak L}, L_{\mathfrak L})$ is injective (its kernel is $\hat{H}^1(G_{\mathfrak L}, {\mathbb Z})$, which is trivial), we conclude that $\hat{H}^2(U_L)$ is $\Delta$-invariant as well.
\end{proof}

We are now ready to connect the cohomological theory back to $M$, the $p$-part of $Cl_L^0$.

\begin{corollary}
$\Delta$ acts trivially on $M[\tau-1]$ and via $\omega$ on $M/(\tau-1)M$.
\end{corollary}

\begin{proof}
By Corollary \ref{iso} and the fact that $N_G$ kills $M$, we have
\[ M[\tau-1] \cong \hat{H}^{0}(Cl_L^0) \cong \hat{H}^1(U_L) \text{ and} \]
\[ M/(\tau-1)M \cong \hat{H}^{-1}(Cl_L^0) \cong \hat{H}^0(U_L) \cong \hat{H}^2(U_L)(\omega), \]
which have the claimed actions by Lemma \ref{action}.
\end{proof}

Finally, we can prove the $p$ part of our result:

\begin{proposition}
$M$ is a $(p-1)$-st power of some finite abelian $p$-group.
\end{proposition}

\begin{proof}
We consider the twist $M' = M(\omega^{-1})$. The previous result says that $\Delta$ acts trivially on $M'/(\tau-1)M'$ and by $\omega^{-1}$ on $M'[\tau-1]$. Thus Proposition \ref{pstruct} may be applied to $M'$. As abelian groups, $M \cong M'$, so we are done.
\end{proof}

Combining this with Proposition \ref{nonp}, we have that the $p$ part and non-$p$ part of $Cl_L^0$ are each the $(p-1)$-st power of an abelian group, and so the proof of Theorem \ref{main} is complete.

\section{Proof of Theorem \ref{main2}}\label{proof2}

\subsection{Character and zeta function relations}

Henceforth let $q$ be a power of an odd prime. The Galois groups $G$, $\Delta$, and $\Omega$ may still be defined as in the beginning of Section \ref{main}, with the caveat that $G$ is no longer a cyclic group when $q$ is not prime. Instead, we have an isomorphism $\nu: G \to {\mathbb F}_q$, where $\nu(\tau)$ is the element of ${\mathbb F}_q$ such that $\tau(\gamma)=\gamma + \nu(\tau) \lambda$.

Now, $\Omega$ can be conveniently realized as the matrix group
$$ \left\{ \left. \begin{pmatrix} a & x \\ 0 & 1 \end{pmatrix} \hspace{.2cm} \right| \hspace{.2cm} a,x \in {\mathbb F}_q, a \neq 0 \right\}, $$
with $\sigma \leftrightarrow \begin{pmatrix} \omega(\sigma) & 0 \\ 0 & 1 \end{pmatrix}$ for $\sigma \in \Delta$ and $\tau \leftrightarrow \begin{pmatrix} 1 & \nu(\tau) \\ 0 & 1 \end{pmatrix}$ for $\tau \in G$.
Thus the elements with $a=1$ are identified with the elements of $G$, and those with $x=0$ with the elements of $\Delta$. 

We are interested in four characters of $\Omega$ which we will show fit an arithmetic relation. These are:
\begin{itemize}
    \item $\chi_L$, for the regular representation (permutation representation on $\Omega$)

    \item $\chi_K$, for the permutation representation on $\Omega / G$
    
    \item $\chi_F$, for the permutation representation on $\Omega / \Delta$
    
    \item $\chi_k$, for the trivial representation (permutation representation on $\Omega / \Omega)$.
\end{itemize}

\begin{proposition}
$$ \chi_L - \chi_k = \chi_K - \chi_k + (q-1)(\chi_F - \chi_k). $$
\end{proposition}

\begin{proof}
We know of course that $\chi_k$ takes the value 1 on every element of $\Omega$, and that $\chi_L$ takes $|\Omega|=q(q-1)$ on the identity and 0 elsewhere.

Now, $\begin{pmatrix} a & x \\ 0 & 1 \end{pmatrix} = \begin{pmatrix} 1 & x \\ 0 & 1 \end{pmatrix} \begin{pmatrix} a & 0 \\ 0 & 1 \end{pmatrix} = \begin{pmatrix} a & 0 \\ 0 & 1 \end{pmatrix} \begin{pmatrix} 1 & a^{-1}x \\ 0 & 1 \end{pmatrix}$. This says that each coset of $\Omega / \Delta$ can be represented by a unique element of $G$, and each coset of $\Omega / G$ by a unique element of $\Delta$.

We have $\begin{pmatrix} a & x \\ 0 & 1 \end{pmatrix} \begin{pmatrix} b & 0 \\ 0 & 1 \end{pmatrix} = \begin{pmatrix} ab & 0 \\ 0 & 1 \end{pmatrix} \begin{pmatrix} 1 & (ab)^{-1}x \\ 0 & 1 \end{pmatrix}$, so an element of $\Omega$ fixes a coset of $\Omega / G$ if and only if $a=1$. This says that $\chi_K \begin{pmatrix} a & x \\ 0 & 1 \end{pmatrix} = |\Omega / G| = q-1$ for $a=1$, and 0 for $a \neq 1$. 

On the other hand, $\begin{pmatrix} a & x \\ 0 & 1 \end{pmatrix} \begin{pmatrix} 1 & y \\ 0 & 1 \end{pmatrix} = \begin{pmatrix} 1 & ax+y \\ 0 & 1 \end{pmatrix} \begin{pmatrix} a & 0 \\ 0 & 1 \end{pmatrix}$, so an element of $\Omega$ fixes a coset of $\Omega / \Delta$ if and only if $x = y(1-a)$. This means that $\chi_F \begin{pmatrix} a & x \\ 0 & 1 \end{pmatrix} = |\Omega / \Delta| = q$ for $a=1, x=0$, 0 for $a=1, x \neq 0$, and 1 for $a \neq 1$.

Using the values ascertained above, the relation holds for each element $\begin{pmatrix} a & x \\ 0 & 1 \end{pmatrix}$ of $\Omega$ as follows:
\begin{itemize}
    \item For $a=1, x=0$: $q(q-1) - 1 = (q-1) - 1 + (q-1)(q-1)$.
    \item For $a=1, x \neq 0$: $0 - 1 = (q-1) - 1 + (q-1)(0 - 1)$.
    \item For $a \neq 1$: $0 - 1 = 0 - 1 + (q-1)(1 - 1)$.
\end{itemize}
\end{proof}

This arithmetic relation between characters gives rise to a corresponding multiplicative relation between L-functions, and thus zeta functions \cite{serre1965zeta}:

\begin{corollary}\label{zeta-rel}
Let $\zeta_*$ denote the zeta function for the field $*$. Then
$$ \frac{\zeta_L}{\zeta_k} = \frac{\zeta_K}{\zeta_k} \cdot \left( \frac{\zeta_F}{\zeta_k} \right)^{q-1}. $$
\end{corollary}

\subsection{Residues of zeta functions}\label{residues}

Schmidt \cite{schmidt1931analytische} gives the residue formula
$$ \lim_{s \to 1} (s-1) \zeta(s) = \frac{q^{1-g}h}{(q-1)\log q} $$
and the functional equation
$$ \zeta(1-s) = q^{(g-1)(2s-1)}\zeta(s) $$
for the zeta function of a function field in positive characteristic, where $g$ and $h$ denote the genus and class number, respectively\footnote{See Roquette \cite{roquette2001class} for an English reference summarizing these results, but note a typo in the numerator of the residue formula (reversing the sign of the exponent).}. These combine to give a formula for the residue of $\zeta(1-s)$ at 1:
$$ \lim_{s \to 1} (s-1)\zeta(1-s) = \lim_{s \to 1} (s-1) q^{(g-1)(2s-1)} \zeta(s) = \frac{h}{(q-1)\log q}. $$

As $K$ and $k$ have trivial class group, applying this equation to the relation in Corollary \ref{zeta-rel} gives
$h_L = h_F^{q-1}$, completing the proof of Theorem \ref{main2}.

\clearpage
\bibliographystyle{plain}
\bibliography{refs}

\begin{thebibliography}{1}

\bibitem{artin1968class}
Emil Artin and John~Torrence Tate.
\newblock {\em Class field theory}, volume 366.
\newblock American Mathematical Soc., 1968.

\bibitem{carlitz1938class}
Leonard Carlitz.
\newblock A class of polynomials.
\newblock {\em Transactions of the American Mathematical Society},
  43(2):167--182, 1938.

\bibitem{goss1983arithmetic}
David Goss.
\newblock The arithmetic of function fields 2: the ‘cyclotomic’ theory.
\newblock {\em Journal of Algebra}, 81(1):107--149, 1983.

\bibitem{hayes1974explicit}
David~R Hayes.
\newblock Explicit class field theory for rational function fields.
\newblock {\em Transactions of the American Mathematical Society}, 189:77--91,
  1974.

\bibitem{honda1971pure}
Taira Honda.
\newblock Pure cubic fields whose class numbers are multiples of three.
\newblock {\em Journal of Number Theory}, 3(1):7--12, 1971.

\bibitem{roquette2001class}
Peter Roquette et~al.
\newblock Class field theory in characteristic $ p $, its origin and
  development.
\newblock In {\em Class field theory--Its centenary and prospect}, pages
  549--631. Mathematical Society of Japan, 2001.

\bibitem{schmidt1931analytische}
Friedrich~Karl Schmidt.
\newblock Analytische zahlentheorie in k{\"o}rpern der charakteristik $p$.
\newblock {\em Mathematische Zeitschrift}, 33(1):1--32, 1931.

\bibitem{schoof2020ideal}
Ren{\'e} Schoof.
\newblock On the ideal class group of the normal closure of {${\mathbf
  Q}(\sqrt[p]{n})$}.
\newblock {\em Journal of Number Theory}, 216:69--82, 2020.

\bibitem{serre1965zeta}
Jean-Pierre Serre.
\newblock Zeta and l functions.
\newblock In {\em Arithmetical Algebraic Geometry, Proc. of a Conference held
  at Purdue Univ., Dec. 5-7, 1963}. Harper and Row, 1965.

\end{thebibliography}

\end{document}